\documentclass[12pt, bezier,amstex, oneside,reqno]{amsart}
\usepackage{amsmath, amssymb, amsthm, verbatim, euscript, amscd, wasysym}
\usepackage[dvips]{graphicx}
\usepackage[all,cmtip]{xy}
\usepackage{epsfig}
\usepackage{color}

\def\ie{{i.e., }}
\def\eg{{e.g., }}

\def\D{Dif\!f_0(M)}
\def\DT{Dif\!f_0(\mathbb T^n)}
\def\X{\EuScript X_L}
\def\HH{Top_0(M)}
\def\HHT{Top_0(\mathbb T^n)}
\def\HHTT{Top_0(\mathbb T^2)}
\def\P{ P}
\def\Ps{{ P}^s}
\def\PP{\EuScript P}
\def\M{\EuScript M}
\def\Z{\mathcal Z}
\def\A{\mathbb A}
\def\N{\EuScript N}
\def\NN{\mathbb N}

\newtheorem{theorem}{Theorem}
\newtheorem{prop}[theorem]{Proposition}
\newtheorem{lemma}[theorem]{Lemma}

\newtheorem*{ass}{Assertion}
\newtheorem{fact}{Fact}
\newtheorem{ffact}{Fact}
\newtheorem*{lemmax}{Lemma~\ref{lemma1}$'$}
\theoremstyle{remark}
\newtheorem*{remarks}{Remarks}
\newtheorem*{remark}{Remark}

\theoremstyle{remark}

\renewcommand\pisces{\textup{H}}
\renewcommand\hat\widehat

\begin{document}
\author{ F. Thomas Farrell and Andrey Gogolev$^\ast$ }
\thanks{$^\ast$Both authors were partially supported by NSF grants.}
\title[The space of Anosov diffeomorphisms]{The space of Anosov diffeomorphisms}
\begin{abstract}
We consider the space $\X$ of Anosov diffeomorphisms homotopic to a fixed automorphism $L$ of an infranilmanifold $M$. We show that if $M$ is the 2-torus $\mathbb T^2$ then $\X$ is homotopy equivalent to $\mathbb T^2$. In contrast, if dimension of $M$ is large enough, we show that $\X$ is rich in homotopy and has infinitely many connected components. 
\end{abstract}

 \maketitle

\section{Introduction}
Let $M$ be a smooth compact $n$-dimensional manifold that supports
an Anosov diffeomorphism. Recall that a diffeomorphism $f\colon M\to
M$ is called {\it Anosov} if there exist constants $\lambda \in
(0,1)$ and $C>0$ along with a $df$-invariant splitting $TM=E^s\oplus
E^u$ of the tangent bundle of $M$, such that for all $m \ge 0$
\begin{multline*}
\qquad\|df^mv\|\le C\lambda^m\|v\|,\;v\in E^s,\; \\
\qquad\shoveleft{\|df^{-m}v\|\le C\lambda^{m}\|v\|,\;v\in E^u.
\hfill}
\end{multline*}

Currently the only known examples of Anosov diffeomorphisms are
Anosov automorphisms of infranilmanifolds and
diffeomorphisms conjugate to them. Furthermore, global structural
stability of Franks and Manning~\cite{F, M} asserts that any Anosov
diffeomorphism $f$ of an infranilmanifold is conjugate to an Anosov
automorphism $L$ with the conjugacy being homotopic to identity.
See, \eg~\cite{KH} for further background on Anosov diffeomorphisms.

In the light of the above discussion we fix an infranilmanifold $M$ and an
Anosov automorphism $L\colon M\to M$. We shall study the space $\X$
of Anosov diffeomorphisms of $M$ that are homotopic to $L$. In other
words, an Anosov diffeomorphism $f$ belongs to $\X$ if and only if
there exists a continuous path of maps $f_t\colon M\to M$ such that
$f_0=L$ and $f_1=f$. If one has a smooth path of diffeomorphisms
(rather than maps) connecting $L$ and $f$ then we say that $f$ is
{\it isotopic} to $L$. We equip $\X$ with $C^r$-topology, $r= 1,2,\ldots \infty$.

Denote by $\D$ the group of diffeomorphisms of $M$ that are
homotopic to identity. We equip $\D$ with $C^r$-topology, where $r\ge 1$ is the same as before. Also denote by $\HH$ the group of
homeomorphisms of $M$ that are homotopic to identity. Equip $\HH$ with
compact-open topology.

The group $\D$ acts on $\X$ by conjugation. It is easy to see that $L$ has a fixed point. Assume for a moment that $M=\mathbb T^n$ and $L$
has only one fixed point. This guarantees, by~\cite{W}, uniqueness of the
conjugacy given by global structural stability. That is, for every
$f\in \X$ there exists unique $h\in \HHT$ such that $f=h\circ L\circ h^{-1}$.
Therefore we have the following inclusions
$$
\hspace{3cm}\DT \hookrightarrow \X \hookrightarrow \HHT
\hspace{3cm}(\ast)
$$
with the composition $\DT\hookrightarrow \HHT$ being the natural
inclusion. Therefore, one gets topological information about the space $\X$ from
that on $\DT$ and $\HHT$. We will make
precise statements and arguments below which are valid for the general
case; i.e., when $M$ is perhaps not $\mathbb T^n$ or when $L$ has possibly more than one
fixed point.

\section{Results}
Our goal is to provide some information on homotopy
type of the space of Anosov diffeomorphisms $\X$. We start by recalling the definition of an Anosov automorphism.

An {\it infranilmanifold} is a double coset space $M\stackrel{\mathrm{def}}{=}G\backslash N\rtimes G/ \Gamma$, where $N$ is a simply
connected nilpotent Lie group, $G$ is a finite group, and $\Gamma$ is 
a torsion-free discrete cocompact subgroup of the semidirect product $N\rtimes G$. When $G$ is trivial $M$ is called
{\it nilmanifold}.
An automorphism $\widetilde L\colon N\to N$ is called {\it hyperbolic} if the
differential $D \widetilde{L}\colon \mathfrak
{n}\to \mathfrak {n}$ does not have eigenvalues of absolute value 1. If an affine map $L\stackrel{\mathrm{def}}{=}v\cdot\widetilde L$ commutes with $\Gamma$ then it induces an affine diffeomorphism $L$ on $M$.
It is easy to show that if $\widetilde L$ is hyperbolic then $L$ is Anosov. And in this case we call $L$ an {\it Anosov automorphism}.

\begin{theorem}
\label{thm_2torus}
Let $L\colon \mathbb T^2\to \mathbb T^2$ be an Anosov automorphism
of the $2$-torus. Then the space $\X$ of $C^r$, $r\ge 1$, Anosov
diffeomorphisms homotopic to $L$ is homotopy equivalent to $\mathbb
T^2$.
\end{theorem}

The proof relies on some standard results and techniques from hyperbolic
dynamics. The outline of the proof is given in Appendix~\ref{appA}.

 Next we collect information about homotopy of $\D$ for higher dimensional $M$. Below
$\mathbb Z_p^\infty$ stands for the direct sum of countably many copies of
 $\mathbb Z_p\stackrel{\mathrm{def}}{=}\mathbb Z/p\mathbb Z$.

\begin{prop}
\label{prop_torus1}
 If  $n\ge 10$, then
$$
\pi_0(\DT)\simeq\mathbb Z_2^\infty\oplus \dbinom{n}{2}\mathbb
Z_2\oplus\sum_{i=0}^{n}\dbinom{n}{i}\Gamma_{i+1},
$$
where $\Gamma_i$, $i=0,\ldots,n$, are the finite abelian groups of Kervaire-Milnor ``exotic"
spheres. Moreover, $\mathbb Z_2^\infty$ maps monomorphically into $\pi_0(\HHT)$ via the map induced by the inclusion $\DT\hookrightarrow\HHT$.
\end{prop}
\begin{proof}
This result is contained in Theorem~4.1 of~\cite{H} and Theorem~2.5 of~\cite{HS} with one caveat. The proofs of both of these theorems depended strongly on a formula given in~\cite{HW} and~\cite{H2}; cf. Theorem~3.1 of~\cite{H}. Igusa found that this formula and its proof were seriously flawed, and he corrected this formula in Theorem~8.$a$.2 of~\cite{I84}. Using Igusa's formula, the two proofs of Proposition~\ref{prop_torus1} mentioned above are valid with minor modifications.\end{proof}
\begin{prop}
\label{prop_torus2}
 Let $p$ be a prime number different from 2 and $k$ be an integer satisfying $2p-4\le k<\frac{n-7}{3}$. Then
$\pi_k(\DT)$ contains a subgroup $S$ such that
\begin{enumerate}
 \item $S\simeq \mathbb Z_p^\infty$ and
\item $S$ maps monomorphically into $\pi_k(\HHT)$ via the map induced by the inclusion $\DT\hookrightarrow\HHT$.
\end{enumerate}
\end{prop}
We postpone the proof of the above proposition to Appendix~\ref{appB}.

\begin{prop} 
\label{prop_nil1}
If $M$ is an infranilmanifold of dimension  $n\ge 10$ then
$$
\mathbb Z_2^\infty < \pi_0(\D).
$$
Moreover, $\mathbb Z_2^\infty$ maps monomorphically into $\pi_0(\HH)$ via the map induced by the inclusion $\D\hookrightarrow\HH$.
\end{prop}
\begin{proof}
This result follows from a slightly augmented form of Proposition~2.2(A) in~\cite{HS} with the same caveat made in the proof of our Proposition~\ref{prop_torus1}. Since $M$ is an infranilmanifold, $\pi_1(M)$ contains a normal nilpotent subgroup $N$ with a finite quotient group $\pi_1(M)/N$. Now note that the center $\Z(N)$ of $N$ is a finitely generated, infinite abelian group. Hence $\Z(\pi_1(M))=\pi_1(Aut(M))$ is also finitely generated (but perhaps not infinite); thus verifying one of the hypotheses of Proposition~2.2(A). And since the $\pi_1(M)$ conjugacy class of any element in $\Z(N)$ is finite, $\pi_1(M)$ contains an infinite number of distinct conjugacy classes; therefore
$$
Wh_1(\pi_1(M);\mathbb Z_2)=\mathbb Z_2^\infty.
$$
Then
$$
Wh_1(\pi_1(M);\mathbb Z_2)/\{c+\varepsilon\bar c\}=H_0(\mathbb Z_2;\mathbb Z_2^\infty)\simeq\mathbb Z_2^\infty
$$
by the simple algebraic argument given on page 287 of~\cite{Far}. That is, we don't need to know whether ``$\pi_1(M)$ contains infinitely many conjugacy classes distinct from their inverse classes" as hypothesized in Proposition~2.2 of~\cite{HS} to complete the argument given in that paper which produces a subgroup $\mathbb Z_2^\infty$ of $\pi_0(\D)$. (The diffeomorphisms representing the elements of $\mathbb Z_2^\infty$ are all homotopic to $id_M$ since they are constructed to be pseudo-isotopic to $id_M$.) Since Proposition~2.2 is also true in the topological category (cf. footnote $(i)$ on page 401 of~\cite{HS}), this subgroup $\mathbb Z_2^\infty$ maps monomorphically into $\pi_0(\HH)$.
\end{proof}

\begin{prop}
\label{prop_nil2}
 Let $M$ be an $n$-dimensional infranilmanifold and $p$ be a prime number different from 2. Assume that the first Betti number of $M$ is non-zero, \ie $H_1(M,\mathbb Q)\neq 0$ and that $n>6p-5$. Then there exists a subgroup $S$ of $\pi_{2p-4}(\D)$ such that
\begin{enumerate}
 \item $S\simeq \mathbb Z_p^\infty $ and
\item $S$ maps monomorphically into $\pi_{2p-4}(\HH)$ via the map induced by the inclusion $\D\hookrightarrow\HH$.
\end{enumerate}
\end{prop}

\begin{remark}
 The first Betti number of any nilmanifold is different from zero.
\end{remark}
We postpone the proof of the above proposition to Appendix~\ref{appB}.

\begin{theorem}

Let $M$ be an $n$-dimensional infranilmanifold, $L\colon M\to M$ be an Anosov automorphism and $\X$ be the space of $C^r$, $r\ge 1$,  Anosov diffeomorphisms homotopic to $L$, then the following is true
\begin{enumerate}
\item If $M=\mathbb T^n$ and $n\ge 10$, then $\X$ has infinitely many connected components.
\item If $M=\mathbb T^n$, $p$ is a prime number different from 2 and $k$ is an integer satisfying $2p-4\le k<\frac{n-7}{3}$, then
$$
\mathbb Z_p^\infty < \pi_k(\X).
$$
\item Let $M$ be an infranilmanifold of dimension $n\ge10$, then
 $\X$ has infinitely many connected components.
\item If $p$ is a prime number different from 2 and $M$ is an infranilmanifold of dimension $n>6p-5$ with a non-zero first Betti number then
$$
\mathbb Z_p^\infty < \pi_{2p-4}(\X).
$$
\end{enumerate}
\end{theorem}
\begin{remark}
 Assertions 1, 2, 3 and 4 above rely on Propositions 2, 3, 4 and 5 respectively. For this reason, even though assertion 1 is implied by assertion 3, we give a separate statement in the torus case. In fact, in the case when $M=\mathbb T^n$ we have more information about the induced map $\pi_0(\D)\to\pi_0(\X)$. In~\cite{FG} we show that this homomorphism is not monic.
\end{remark}

\begin{proof} We prove the fourth and the third assertions only. The proofs of the
other assertions are the same. 

We begin with the proof of assertion 4. Let $id_M$ and $L$ be the base-points in $\D$ and $\X$, respectively. Pick two different elements $[\gamma_1], [\gamma_2]\in \mathbb Z_p^\infty
< \pi_{2p-4}(\D)$. To prove assertion 4 we need to show that the loops $\gamma_1\circ L\circ \gamma_1^{-1}$ and $\gamma_2\circ L\circ \gamma_2^{-1}$ are different in $\pi_{2p-4}(\X)$. Let $\alpha=\gamma_1\circ\gamma_2^{-1}$ and $\beta=\alpha\circ L\circ\alpha^{-1}$. Clearly $[\alpha]$ is non-trivial in $\pi_{2p-4}(\D)$ and it is enough to show that $[\beta]$ is non-trivial in $\pi_{2p-4}(\X)$.

Assume that there exists a homotopy $\beta_t$, $t\in[0,1]$, such that $\beta_0=\beta$ and $\beta_1=L$. Then structural stability yields a homotopy $\alpha_t$, $t\in [0,1]$ such that $\alpha_0=\alpha$ and $\alpha_1$ commutes with $L$, \ie $\alpha_1(x)\circ L=L\circ\alpha_1(x)$ for all $x\in \mathbb D^{2p-4}$. Hence, by local uniqueness part of structural stability theorem, $\alpha_1$ is constant. On the other hand, we know that $\alpha_0|_{\partial \mathbb D^{2p-4}}=id_M$ and hence, again by local uniqueness, $\alpha_t|_{\partial \mathbb D^{2p-4}}=id_M$ for all $t\in[0,1]$. Hence $\alpha_1=id_M$ and we conclude that $[\alpha]=[id_M]$, which is a contradiction. 

\begin{figure}[htbp]
\begin{center}
\includegraphics{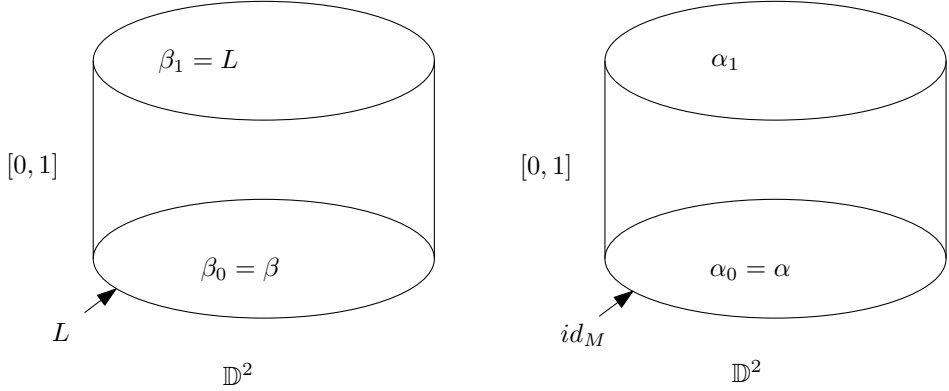}
\end{center}
 \caption{Homotopies of $\beta$ and $\alpha$ when $p=3$.}
\label{fig1}
\end{figure}

The proof of assertion 3 is more involved because we have to deal with the centralizer of $L$.
We will rely on the following proposition whose proof we postpone to Appendix~\ref{appC}.
\begin{prop} \label{prop_translation} Let $M$ be an infranilmanifold. Assume that a homeomorphism $h\colon M\to M$ is homotopic to identity and commutes with an Anosov automorhism $L\colon M\to M$. Then $h$ is isotopic to $id_M$.
\end{prop}
Take $h_1, h_2 \in \D$ that
represent different elements $[h_1]$, $[h_2]$ of $\mathbb Z_2^\infty
< \pi_0(\D)$. Let $f_1=h_1\circ L\circ h^{-1}_1$ and $f_2=h_2\circ L\circ h^{-1}_2$. Assume
that there is a path of Anosov diffeomorphisms $f_t$, $t\in [1, 2]$,
connecting $f_1$ and $f_2$. By structural stability and local
uniqueness we get a continuous path $\{\tilde{h}_t,
t\in[1, 2]\}$ in $\HH$ such that $\tilde{h}_1=h_1$ and
$f_t=\tilde{h}_t\circ L\circ \tilde{h}_t^{-1}$. Hence we get
$$
h_2\circ L\circ h_2^{-1}=\tilde{h}_2\circ L\circ \tilde{h}_2^{-1}.
$$
The homeomorphism $\tilde{h}_2^{-1}\circ h_2$ commutes with
$L$. Therefore Proposition~\ref{prop_translation} implies that
$[\tilde{h}_2^{-1}\circ h_2]=[id_M]$. Therefore
$$
[h_2]=[\tilde{h}_2]=[h_1] \textup{ in } \pi_0(\HH),
$$
which gives us a contradiction. We conclude that $f_1$ and $f_2$
represent different connected components of $\X$.
\end{proof}

\begin{remarks}$ $
\begin{enumerate}
\item It is not clear whether or not there are other connected
components of $\X$ that we are not detecting. {\it Question:} for which $h\in\D$, does the
connected component of $h\circ L$ in $Dif\!f(M)$ contain an
Anosov diffeomorphism?
\item By Moser's homotopy trick the space $\D^{vol}$ consisting of all volume preserving diffeomorphisms homotopic to $id_M$ is a deformation
retraction of $\D$. Hence, we have analogous results for the space
of volume preserving Anosov diffeomorphisms $\X^{vol}$.
\item See page 10 of~\cite{H} for a conjectural geometric description for representatives of non-zero elements in $\mathbb Z_2^\infty < \pi_0(\D)$ of Propositions~\ref{prop_torus1} and~\ref{prop_nil1}.
\end{enumerate}
\end{remarks}

{\bf Acknowledgements.} We would like to thank Pedro Ontaneda, Rafael Potrie and the referee for useful comments which lead to better presentation.

\appendix
\section{Proof of Theorem~\ref{thm_2torus}}
\label{appA}
\subsection{Outline of the proof of Theorem~\ref{thm_2torus}}
\subsubsection{Step 1: the 2-torus}
Let $p_0$ be a fixed point of $L$, $L(p_0)=p_0$. For every point
$p\in \mathbb T^2$ consider the translation $t_p\colon \mathbb T^2\to
\mathbb T^2$ on the (flat) torus that takes $p_0$ to $p$. Then
$L_p\stackrel{\mathrm{def}}{=}t_p\circ L\circ t_p^{-1}$ is an Anosov
automorphism that fixes point $p$. 

Consider the space $\mathcal T^2=\{L_p: p\in\mathbb T^2\}$ equipped with $C^r$ topology. Then it is easy to see that the map $\mathbb T^2\ni p\mapsto L_p\in\mathcal T^2$ is a finite covering and that $\mathcal T^2$ is a 2-torus. Our goal is to show that $\X$ is homotopy equivalent to $\mathcal T^2\subset\X$.

\subsubsection{Step 2: $\X$ has homotopy type of a CW complex}
Recall that the space of all $C^r$ diffeomorphisms of $\mathbb T^2$ --- $Dif\!f(\mathbb T^2)$ --- is a separable infinite dimensional manifold modeled on the Banach space or the Fr\'echet space of $C^r$ vector fields when $r<\infty$ or $r=\infty$, respectively. Hence $Dif\!f(\mathbb T^2)$ is a separable absolute neighborhood retract. Since $\X$ is an open subset of $Dif\!f(\mathbb T^2)$, we also have that $\X$ is a separable absolute neighborhood retract. By a result of W.H.C. Whitehead~\cite{P}, every absolute neighborhood retract has the homotopy type of a CW complex. Therefore $\X$ has homotopy type of a CW complex.

\subsubsection{Step 3: $\X\simeq\mathcal T^2$}
\begin{lemma}
\label{lemma8}
Let $\mathbb D^k$ be
a disk of dimension $k$ and let $\alpha\colon (\mathbb D^k,\partial)\to (\X,L) $ be a
continuous map which sends the boundary of the disk to $L$, \ie $[\alpha]\in \pi_k(\X)$. Then $\alpha$ can be homotoped to $\widehat{\alpha}\colon (\mathbb D^k,\partial)\to
(\mathcal T^2, L)$. 
\end{lemma}
This lemma implies that the inclusion $\mathcal T^2\hookrightarrow \X$ induces an epimorphism on homotopy groups. Therefore, since $\mathcal T^2$ is aspherical, $\pi_k(\mathcal T^2)\to \pi_k(\X)$ is a trivial isomorphism for $k\ge 2$. The homomorphism $\pi_1(\mathcal T^2)\to\pi_1(\X)$ is monic and, hence, is an isomorphism as well. The fact that the homomorphism $\pi_1(\mathcal T^2)\to\pi_1(\X)$ is monic can be deduced from structural stability and the fact that $\pi_1(\mathbb T^2)\to\pi_1(\HHTT)$ is monic. Now, since $\X$ has homotopy type of a CW complex, Theorem~\ref{thm_2torus} follows from J.H.C. Whitehead's Theorem.

We proceed with preparations to and then a sketch of the proof of Lemma~\ref{lemma8}. 
\subsection{Convention.} When we say that
an object is $C^{1+}$ we mean that it is $C^1$ and the first
derivative is H\"{o}lder continuous with some positive exponent.

\subsection{Equilibrium states for Anosov diffeomorphisms on $\mathbb T^2$.}
Let $g\colon \mathbb T^2\to \mathbb T^2$ be an Anosov diffeomorphism. We denote by
$W^s(g)$ and $W^u(g)$ the stable and unstable foliations of $g$. We
assume that a background Riemannian metric $a$ is fixed.
The logarithms of stable and unstable jacobians of $g$ will be
denoted by $\varphi^s(g)$ and $\varphi^u(g)$. 

Two H\"{o}lder
continuous functions $\varphi_1, \varphi_2\colon \mathbb T^2\to
\mathbb R$ are called {\it cohomologous up to an additive constant over $g$} if there exist a
constant $C$ and a H\"{o}lder continuous function $u\colon\mathbb
T^2\to \mathbb R$ such that $\varphi_1 = \varphi_2+C+u\circ g-u$. In
this case we write
$\langle\varphi_1\rangle=\langle\varphi_2\rangle$.  We remark that even though  $\varphi^s(g)$ and $\varphi^u(g)$ depend on the chosen Riemannian metric $a$, the cohomology classes  $\langle\varphi^s(g)\rangle$ and $\langle\varphi^u(g)\rangle$ are independent of the choice of $a$.

Let $T$ be a transveral to the unstable foliation $W^u(g)$. For each $y\in T$ consider a finite arc $V^u(y)\subset W^u(g)(y)$ that intersect $T$ at $y$. Assume that $V^u(y)$ vary continuously with $y$. Then  given  a H\"{o}lder continuous potential $\varphi\colon\mathbb T^2\to\mathbb R$ where
exists a unique system of measures $\{\mu_y^{\varphi}, y\in T\}$,
satisfying the following properties.
\begin{enumerate}
\item[(P1)] $\mu_y^{\varphi}, y\in T$, are finite measures
supported on $V^u(y)$; \vspace{0.3cm}
\item[(P2)] If for $z\in V^u(y)$ the point $f(z)\in V^u(\bar y)$ then $\dfrac{d(f^*\mu_{\bar y}^\varphi)}{d\mu_y^\varphi}=
e^{\varphi(z)-P(\varphi)}$; here $P(\varphi)$ is the pressure of
$\varphi$;\vspace{.2cm}
\item[(P3)] the system $\{\mu_y^\varphi\}$ satisfies certain absolute continuity property with respect to the stable foliation $W^s(f)$.
\end{enumerate}
Measures $\mu_y^\varphi$ are equivalent to the conditional measures
on $V^u(y)$ of the equilibrium state of $\varphi$. If
$\varphi=\varphi^u(g)$ then $\mu_y^\varphi$ are absolutely
continuous measures induced by the Riemannian metric $a$.
For more details about the system $\{\mu_y^\varphi\}$ and the proof
of existence and uniqueness, see, \eg~\cite{L}.

\subsection{Sketch of the proof of Lemma~\ref{lemma8}}
Let $\xi(x_0)=id_{\mathbb T^2}$ for some $x_0\in\partial \mathbb D^k$. Then structural stability gives the continuation $\xi\colon \mathbb D^k\to
\HHTT$ to the rest of the disk of the conjugacy to the linear
model:
$$
\alpha (\cdot)\circ \xi(\cdot)=\xi(\cdot)\circ L.
$$
If $k\ge 2$ then, by local uniqueness part of structural stability, $\xi(x)=id_{\mathbb T^2}$ for $x\in\partial\mathbb D^k$.
In this case we define the continuation of the
fixed point $p_0$ to the interior of the disk by the formula $fix(\cdot)\stackrel{\mathrm{def}}{=}\xi(\cdot)(p_0)$. Then clearly $\alpha (\cdot)(
fix(\cdot))=fix(\cdot)$. If $k=1$ then $\xi$ is $id_{\mathbb T^2}$ at the endpoint $x_0$ and a possibly non-trivial translation at the other endpoint. In this case, $fix$ defined as before gives a path that connects $p_0$ with a fixed point of $L$. We shall explain how to construct a homotopy that connects $\alpha$ to $\widehat{\alpha}\stackrel{\mathrm{def}}{=}L_{fix(\cdot)}$. (Note that $L_{fix(\cdot)}$ is always a loop, even when $fix$ is a path from one fixed point to another fixed point.)

\begin{remark}
 From now on we will assume that $\alpha$ consists of $C^{1+}$ diffeomorphisms. In case when $\X$ is the space of $C^1$ Anosov diffeomorphisms we replace given $\alpha$ with a $C^{1+}$ version by approximating and performing a short homotopy. 
\end{remark}

Take $x\in \mathbb D^k$ and let $f\stackrel{\mathrm{def}}{=}\alpha(x)$. Point $p\stackrel{\mathrm{def}}{=}fix(x)$ is a fixed point of $f$. Next we
construct a path of diffeomorphisms that connects $f$ and $f_p$, where $f_p$ is a $C^{1+}$ Anosov
diffeomorphism $C^{1+}$ conjugate to $L_p$. The path will consist of
Anosov diffeomorphisms of regularity $C^{1+}$ that fix $p$.

Choose a simple closed curve $T$ which is transverse to $W^u(f)$ and passes
through $p$. Transversal $T$ cuts the leaves of the unstable foliation $W^u(f)$ into
oriented arcs $[y, e(y)]$ parametrized by $y\in T$.

Fix a background Riemannian metric $a$. Consider the path of potentials $\varphi_t\stackrel{\mathrm{def}}{=}(1-2t)\varphi^u(f)$,
$t\in [0, 1/2]$ and corresponding system of measures
$\mu_y^{\varphi_t}$ on $[y, e(y)]$ as described in Subsection A.3. These measures depends continuously on $t$ (see, \eg~\cite{C}).

Now we can define a $C^{1+}$ path of Anosov diffeomorphisms whose
logarithmic unstable jacobians are cohomologous up to a constant to
$\varphi_t$. This is done in the following way.

Consider the functions $\eta_t\colon T\to \mathbb R$ given by $\eta_t(y)=\mu_y^{\varphi_t}([y, e(y)])$.
Choose a continuous family of Riemannian metrics $a_t$, $t\in [0,
1/2]$, such that  $a_0 = a$ and the induced lengths $l_y^t([y, e(y)])$
of the arcs $[y, e(y)]$ in the metric $a_t$ equal to $\eta_t(y)$.
Consider the family of homeomorphisms $h_t\colon\mathbb T^2\to\mathbb T^2$ that preserve the partition by
the arcs $[y, e(y)]$, $y\in T$, and satisfy the following relation
$$
\mu_y^{\varphi_t}([y, z])=l_y^t([y, h_t(z)]), z\in [y, e(y)].
$$
Clearly, the family of homeomorphisms $h_t$ is uniquely determined. Because $\varphi_0=\varphi^u(f)$ homeomorphism $h_0=id_{\mathbb T^2}$.
Define
$$
f_t \stackrel{\mathrm{def}}{=}h_t\circ f\circ h_t^{-1}, t\in [0, 1/2].
$$
Then $f_0=f$. One can use property (P2) to check that $f_t$, $t\in [0, 1/2]$,
are $C^{1+}$-differentiable along $W^u(f)$. Also one can use property (P3) to check that $h_t(W^s(f))$ are $C^{1+}$ foliations for $t\in[0,1/2]$. Therefore $f_t$, $t\in [0, 1/2]$, are $C^{1+}$ Anosov diffeomorphisms with
$\langle\varphi^u(f_t)\rangle=\langle\varphi_t\circ h_t^{-1}\rangle$
over $f_t$. Because the stable foliation $W^s(f)$ is $C^{1+}$, it follows that
$\langle\varphi^s(f_t)\rangle = \langle\varphi^s(f)\circ
h_t^{-1}\rangle$, $t\in [0, 1/2]$.

Now we switch the roles of the stable and the unstable foliations and apply the same construction to $f_{1/2}$ to get a path $f_t$, $t\in[1/2, 1]$,
connecting $f_{1/2}$ to $f_1$. Then
$\langle\varphi^s(f_1)\rangle=\langle0\rangle$,
$\langle\varphi^s(f_1)\rangle=\langle\varphi^s(f_{1/2})\rangle=\langle0\rangle$. Then it follows that
$f_p\stackrel{\mathrm{def}}{=}f_1$ is $C^{1+}$ conjugate to $L_p$ by~\cite{dlL}.

It is routine to check that the construction outlined above can be
carried out simultaneously for all $\alpha(x)$, $x\in\mathbb D^k$, and that the resulting homotopy can be made to be constant on $\partial \mathbb D^k$. The
choices of transversals $T=T(x)$ and families of Riemannian metrics $a_t=a_t(x)$
must be made continuously in $x$ to make sure that a (continuous) homotopy
of $\alpha$ is produced. This homotopy connects $\alpha$ and
$\Bar{\alpha}\colon (\mathbb D^k,\partial) \to (\textup{Diff}^{1+}(\mathbb T^2), L)$ whose
image lies in $C^{1+}$ conjugacy class of $L$. Using standard
smoothing methods we can $C^1$ approximate our homotopy by another
one that takes values in $\X$ and connects $\alpha$ to
$\widetilde{\alpha}\colon (\mathbb D^k,\partial) \to (\X,L)$. 

The map $\Bar\alpha$ can be $C^1$ approximated by a map whose image lies in the $C^r$ conjugacy class of $L$ simply by approximating $C^{1+}$ conjugacy with a $C^r$ conjugacy. This map is $C^1$ close to $\widetilde\alpha$ and hence, by performing a short homotopy if needed, we can assume that the map $\widetilde\alpha$ is $C^r$ conjugate to $L$.

Finally, map
$\widetilde{\alpha}$ can be homotoped to the map
$\widehat{\alpha}\colon (\mathbb D^k, \partial) \to (\mathcal T^2, L)$ by homotoping
corresponding map $h\colon \mathbb D^k\to \D$, $h(\cdot)\circ L\circ
h^{-1}(\cdot)=\widetilde{\alpha}(\cdot)$, in the space of
$C^r$ conjugacies to a map consisting of the translations $t\colon \mathbb D^k\to \D$ given by $t(x)=t_{fix(\widetilde{\alpha}(x))}=t_{fix(\alpha(x))}$. The latter is
possible due to a result of Earle and Eells \cite{EE} who showed
that $\mathbb T^2=\{t_p, p\in \mathbb T^2\}$ is a deformation
retraction of the space of smooth conjugacies $Dif\!f_0(\mathbb T^2)$.

\section{Proofs of Propositions~\ref{prop_torus2} and~\ref{prop_nil2}}
\label{appB}
\begin{remark} We write $k\ll n$ for $3k+7<n$; in particular, $2p-4 \ll n$ if and only if $6p-5<n$.
 \end{remark}

\begin{proof}[Proof of Proposition~\ref{prop_nil2}]
Consider the following commutative diagram
$$
\begin{CD}
\Ps(T)@>>>\Ps(M)@>t>>\D \\
@VVV@VVV@VVV\\
\P(T)@>>>\P(M)@>t>>\HH
\end{CD}
\hspace{3cm}(\star)
$$
where $T$ is a closed tubular neighborhood of a smooth simple closed curve $\alpha$ in $M$ such that the homology class represented by $\alpha$ generates an infinite cyclic direct summand of $H_1(M)$.
\begin{remark}
If $M$ is orientable then $T=S^1\times\mathbb D^{n-1}$. In general it is the mapping torus of a self-diffeomorphism of $\mathbb D^{n-1}$.
\end{remark}
In this diagram $\Ps(\cdot)$ and $\P(\cdot)$ are the smooth and topological pseudo-isotopy functors, respectively. Recall that a topological (smooth) pseudo-isotopy of a compact manifold $M$ is a homeomorphism (diffeomorphism)
$$
f\colon M\times [0,1]\to M\times [0,1]
$$
such that $f(x)=x$ for all $x\in M\times 0\cup\partial M\times[0,1]$. Then $\P(M)$ ($\Ps(M)$) is the topological space consisting of all such homeomorphisms (diffeomorphisms), respectively. Since $T\subset M$ is a codimension 0 submanifold, a pseudo-isotopy $f$ of $T$ canonically induces a pseudo-isotopy $F$ of $M$ by setting $F(x)=f(x)$,
when $x\in T\times[0,1]$, and $F(x)=x$ otherwise.

Note that the pseudo-isotopy $f$ must map $M\times 1$ into itself. Hence, after identifying $M$ with $M\times 1$ in the obvious way, the restriction of $f$ to $M\times 1$ determines an element in $\HH$ or $\D$ depending on whether $f\in\P(M)$ or $\Ps(M)$. This restriction gives the maps $t$ (standing for top) in diagram~$(\star)$.

Now Proposition~\ref{prop_nil2} is clearly implied by the following Assertion.
\begin{ass}
Let $M$ be an $n$-dimensional infranilmanifold and $p$ be a prime number different from 2. Assume
that the first Betti number of $M$ is non-zero and that $n>6p-5$. Then there exists a subgroup $S$ of $\pi_{2p-4}(\Ps(T))$ such that
\begin{enumerate}
 \item $S\simeq \mathbb Z_p^\infty $ and
\item $S$ maps into $\pi_{2p-4}(\HH)$ with a finite kernel via the homomorphism which is functorially induced by the maps in the commutative diagram~$(\star)$.
\end{enumerate}
\end{ass}
This Assertion will be proven by concatenating several facts which we now list.
\begin{fact}
\label{fact1}
 The kernel of the homomorphism $\pi_k\Ps(T)\to\pi_k\P(T)$ is a finitely generated abelian group provided $k\ll n$.
\end{fact}
This fact follows from Corollary~4.2 in~\cite{FO}.
\begin{fact}
\label{fact2}
Denote the inclusion map $T\subset M$ by $\sigma$. Then the induced homomorphism $\pi_k\P(\sigma)\colon\pi_k\P(T)\to\pi_k\P(M)$ is monic provided $k\ll n$.
\end{fact}
Fact~\ref{fact2} is proven as follows. Since the class of $\alpha$ in $H_1(M)$ generates an infinite cyclic direct summand, there clearly exists a continuous map $\gamma\colon M\to S^1$ such that the composition
$$
S^1\stackrel{\alpha}{\to}T\stackrel{\sigma}{\to}M\stackrel{\gamma}{\to}S^1
$$
is homotopic to $id_{S^1}$. Let $\PP(\cdot)$ denote the stable topological pseudo-isotopy functor. Applying  $\PP(\cdot)$ to the above composition yields that
$$
\pi_k\PP(\sigma)\colon\pi_k\PP(T)\to\pi_k\PP(M)
$$
is monic since $\PP(\cdot)$ is a homotopy functor; cf.~\cite{H}. Therefore Igusa's stability result~\cite{I} completes the proof of Fact~2.

There is an involution ``\---'' defined on $\P(M)$ which is essentially determined by ``turning a pseudo-isotopy upside down.'' See pages~6 and~18 of~\cite{H} for a precise definition. (Also see page 298 of~\cite{FO}.) Since $\sigma$ commutes with ``\---'', it induces a homomorphism
$$
H_0(\mathbb Z_2, \pi_k\P(T))\to H_0(\mathbb Z_2,\pi_k\P(M)).
$$ 
\begin{fact}\label{fact3}
This homomorphism is monic provided $k\ll n$.
\end{fact}
Fact~\ref{fact3} is proven by an argument similar to that given for Fact~\ref{fact2}
\begin{fact}\label{fact4}
If $k\ll n$, then $\pi_k\P(t)\colon\pi_k\P(M)\to\pi_k\HH$ factors through a homomorphism
$$
\varphi\colon H_0(\mathbb Z_2;\pi_k\P(M))\to\pi_k\HH,
$$
whose kernel contains only elements of order a power of 2.
\end{fact}
Fact~\ref{fact4} follows from Hatcher's spectral sequence (see pages 6 and 7 of~\cite{H}) by using that topological rigidity holds for all infranilmanifolds (proven in~\cite{FH}). The argument is a straightforward extension of the one proving Corollary~5 in Section 5 of~\cite{Far}.
\begin{fact}
\label{fact5}
There is a subgroup $S$ of $\pi_k\Ps(T)$, where $k=2p-4\ll n$, satisfying 
\begin{enumerate}
 \item $S\simeq\mathbb Z_p^\infty$ and
\item both $x\mapsto x+\bar x$ and $x\mapsto x-\bar x$ are monomorphisms of $S$ into $\pi_k\Ps(T)$.
\end{enumerate}
\end{fact}

When $T$ is orientable, \ie $T=S^1\times \mathbb D^{n-1}$, then Fact~\ref{fact5} follows from Proposition~4.6 of~\cite{FO} --- which is the analogous result valid for $\pi_k\PP^s(S^1)$ --- by again using Igusa's stability theorem~\cite{I}. We note that Proposition~4.6 depended on important calculations of $\pi_k\PP^s(S^1)$ which can be found in~\cite{GKM}.

In the non-orientable case, $T=\M\times\mathbb D^{n-2}$ where $\M$ denotes the M\"obius band and we argue as follows. Since $S^1\times\mathbb D^1$ is a collar neighborhood of the boundary $\partial\M$, we can also identify $S^1\times\mathbb D^{n-1}$ with a collar neighborhood of $\partial T$. There are two natural maps $S^1\times\mathbb D^{n-1}\to T$; namely, the inclusion map $\omega$  of the collar neighborhood and the (2-sheeted) orientation covering map $q\colon S^1\times\mathbb D^{n-1}\to T$ which induces a transfer map
$$
q^*\colon\Ps(T)\to\Ps(S^1\times\mathbb D^{n-1}).
$$
And let
$$
\omega_*\colon\Ps(S^1\times\mathbb D^{n-1})\to\Ps(T)
$$
denote the natural (induction) map previously denoted by $\Ps(\omega)$. A pleasant exercise shows that
$$
\pi_k(q^*\circ\omega_*)\colon\pi_k\Ps(S^1\times\mathbb D^{n-1})\to\pi_k\Ps(S^1\times\mathbb D^{n-1})
$$
is multiplication by 2. Hence the kernel of
$$
\pi_k(\omega_{\ast})\colon\pi_k\Ps(S^1\times\mathbb D^{n-1})\to\pi_k\Ps(T)
$$
contains only 2-torsion. Denote $\pi_k(\omega_{\ast})$ by $\Omega$ and let $S'$ be a subgroup of $\pi_k(\Ps(S^1\times\mathbb D^{n-1}))$ satisfying properties~1 and 2 of Fact~\ref{fact5}. Then $S\stackrel{\mathrm{def}}{=}\Omega(S')$ is a subgroup of $\pi_k\Ps(T)$ which clearly satisfies property~1 since $p\neq 2$. To see that property~2 is also satisfied consider the homomorphism $\Phi\colon S'\to\pi_k\Ps(T)$ defined by
$$
\Phi(y)\stackrel{\mathrm{def}}{=}\Omega(y+\bar y),
$$ 
and observe that $\Phi$ is monic since the homomorphism
$$
y\mapsto y+\bar y,\;\; y\in S'
$$
is monic and its image contains only $p$-torsion ($p\neq 2$). But
$$
\Phi(y)=\Omega(y)+\overline{\Omega(y)}=x+\bar x,
$$
where $x=\Omega(y)$ is an arbitrary element of $S$. Consequently the homomorphism $x\mapsto x+\bar x$, $x\in S$ is also monic. A completely analogous argument shows that
$$
x\mapsto x-\bar x, \;\;x\in S
$$
is monic completing the verification of property 2.

We now string together these facts to prove the Assertion. The group $S$ of the Assertion is the one given in Fact~\ref{fact5}. Applying the functor $\pi_k(\cdot)$ to diagram~$(\star)$ yields the following commutative diagram 
$$
\xymatrix{
S\subset\!\!\!\!\!\!\!\!\!\!\!\!\!\!\!\!\!\!\!\!\!\!\!\!\!\!\!\!\!\!\!\!\!\!\!\!\!&\pi_k\Ps(T) \ar[r] \ar[d]_\eta &\pi_k\Ps(M) \ar[r] \ar[d]_\pisces &\pi_k\D \ar[d]\\
&\pi_k\P(T) \ar[r] \ar[d]_\psi&\pi_k\P(M) \ar[d]_\Psi \ar[r]^{\pi_k(\P(t))}&\pi_k\HH&(\star\star) \\
&H_0(\mathbb Z_2;\pi_k\P(T))\ar[r]^{\varkappa} &H_0(\mathbb Z_2;\pi_k\P(M))\ar[ru]_\varphi}
$$
The triangle in diagram~$(\star\star)$ is the factorization of $\pi_k(\P(t))$ given in Fact~\ref{fact4}. And the bottom vertical maps are the quotient homomorphisms
$$
\pi_k\P(X)\to\pi_k\P(X)/ \{x-\bar x\colon x\in\pi_k\P(X)\},
$$
where $X=T$ and $M$, respectively. Also we denote by $\eta$ the homomorphism studied in Fact~\ref{fact1}, and by $\varkappa$ the monomorphism in Fact~\ref{fact3}.

Combining Facts~\ref{fact1} and~\ref{fact5}, we see that the kernel of $x\mapsto\eta(x+\bar x)$, $x\in S$, is a finite group. Consequently, the kernel of
$$
\psi\circ \eta\colon S\to H_0(\mathbb Z_2;\pi_k P(T))
$$
is also a finite group. And since $\psi\circ\eta(S)$ is a $p$-torsion group, where $p\neq 2$, we see that
$$
\varphi\circ\varkappa\colon \textup{image}(\psi\circ\eta)\to\pi_k\HH
$$
is monic by Facts~\ref{fact3} and~\ref{fact4}. Therefore the composition
$$
\varphi\circ\varkappa\circ\psi\circ\eta\colon S\to\pi_k\HH
$$
has finite kernel. But this is the homomorphism of part 2 of the Assertion.
\end{proof}

\begin{proof}[Proof of Proposition~\ref{prop_torus2}]
Let $X\mapsto A(X)$ denote Waldhausen's algebraic $K$-theory of spaces functor defined in~\cite{Wald}. We start with the following result.
\begin{lemma}
\label{lemma1}
 For every prime $p\neq 2$ and every integer $k\in [2p-4, (2p-4)+n-1]$, the group $\pi_kA(\mathbb T^n)$ contains a subgroup $\mathbb Z_p^\infty$ such that the following two group endomorphisms
$$
x\mapsto x+\bar x\;\;\;\;\;\;\;\mbox{and}\;\;\;\;\;\; x\mapsto x-\bar x
$$
are both monic when restricted to $\mathbb Z_p^\infty$.
\end{lemma}
\begin{proof}
We verify this by induction on $n$. For $n=1$ it was verified in the proof of Proposition~4.6 from~\cite{FO}. Now
assume that Lemma~\ref{lemma1} is true for $n$, we proceed to verify it for $n+1$. Since $\mathbb T^{n+1}=\mathbb T^n\times S^1$, $\pi_kA(\mathbb T^{n+1})$ contains as subgroups both $\pi_kA(\mathbb T^n)$ and $\pi_{k-1}A(\mathbb T^n)$ in an involution consistent way; cf.~\cite{HKVWW}.
\end{proof}
Since Waldhausen proved in~\cite{Wald} that the kernel of $\pi_kA(X)\to\pi_k\Ps(X)$ is finitely generated, Igusa's stability theorem~\cite{I} yields the following variant of Lemma~\ref{lemma1}.
\begin{lemmax} Let $p$ be a prime number different from 2 and $k$ be an integer such that $2p-4\le k<\frac{n-7}{3}$. Then $\pi_k\Ps(\mathbb T^n)$ contains a subgroup $\mathbb Z_p^\infty$ such that the following two group endomorphisms
$$
x\mapsto x+\bar x\;\;\;\;\;\;\;\mbox{and}\;\;\;\;\;\; x\mapsto x-\bar x
$$
are both monic when restricted to $\mathbb Z_p^\infty$.
\end{lemmax}

Now we follow the pattern used to prove Proposition~\ref{prop_nil2} except that the argument is now simpler since the first column in both diagrams~$(\star)$ and~$(\star\star)$ can be omitted. It clearly suffices to show that the subgroup $\mathbb Z_p^\infty$ of $\pi_k\Ps(\mathbb T^n)$, given by Lemma~\ref{lemma1}$'$ maps with finite kernel into $\pi_k\HHT$ via composite homomorphism $\varphi\circ\Psi\circ\pisces$. Since the kernel of
$$
\pisces\colon\pi_k\Ps(\mathbb T^n)\to\pi_k\P(M)
$$
is finitely generated by Corollary~4.2 of~\cite{FO}, the kernel of $\pisces|_{\mathbb Z_p^\infty}$ is finite. Now arguing as in the proof of Proposition~\ref{prop_nil2}, we see that the kernel of
$\Psi\circ\pisces:\mathbb Z_p^\infty\to H_0(\mathbb Z_2;\pi_k\P(\mathbb T^n)$
is also finite. (Here we crucially use the fact from Lemma~\ref{lemma1}$'$ that $x\mapsto x\pm\bar x $ is monic for $x\in\mathbb Z_p$.) Finally, we conclude from Fact~\ref{fact4} that
$$
\varphi\circ(\Psi\circ\pisces)\colon\mathbb Z_p\to\pi_k\HHT
$$
has finite kernel since $p\neq2$.
\end{proof}

\section{Proof of Proposition~\ref{prop_translation}}
\label{appC}
The universal cover of the infranilmanifold $M$ is a simply connected nilpotent Lie group $N$. The fundamental group $\pi_1(M)$ acts freely on the right by affine diffeomorphisms. Group $N$ acts on itself by left translations. Let
$$
\Gamma\stackrel{\mathrm{def}}{=}\pi_1(M)\cap N.
$$
It is well known that 
$$
\hat M\stackrel{\mathrm{def}}{=}N/\Gamma
$$
is a compact nilmanifold and
$$
G\stackrel{\mathrm{def}}{=}\pi_1(M)/\Gamma\subset\A
$$
is a finite subgroup of {\it the group $\A$ of all affine diffeomorphisms of $\hat M$}. Group $G$ acts freely on $\hat M$; the orbit space of this action is the infranilmanifold $M$.
\begin{ffact}\label{ffact1}
The centralizer $N^{\pi_1M}\stackrel{\mathrm{def}}{=}\{x\in N :\;\; axa^{-1}=x\;\;\mbox{for all}\; a\in\pi_1(M)\}$ is path connected.
\end{ffact}
\begin{proof}
Let $x\in N^{\pi_1M}$. Since $N$ is simply connected nilpotent Lie group, Mal$'$cev's work~\cite{Malcev} yields a $1$-parameter subgroup $\alpha\colon\mathbb R\to N$ such that $\alpha(1)=x$. Let $a\in\pi_1(M)$, then
$$
\beta(s)\stackrel{\mathrm{def}}{=}a\alpha(s)a^{-1}
$$
is also a 1-parameter subgroup such that $\beta(1)=a\alpha(1)a^{-1}=axa^{-1}=x$. Hence $\alpha(s)=\beta(s)$ for all $s\in\mathbb R$, again by Mal$'$cev's work. Therefore $\alpha(s)\in N^{\pi_1M}$ for all $s\in\mathbb R$, yielding that $N^{\pi_1M}$ is connected.
\end{proof}

Group $N$ also acts on $\hat M$ by left translations. In this way
$$
N\to\A\subset Dif\!f(\hat M)
$$
and its image, denoted by $\N$, is a Lie group called {\it the group of translations of $\hat M$}. Since the kernel of $N\to\A$ is $\Z(N)\cap\Gamma$ we have that
$$
\N=N/\Z(N)\cap \Gamma.
$$
Here $\Z(N)$ stands for the center of $N$.

Let $\NN=\NN(\N,G)$ denote the normalizer of $G$ inside $\N$.
\begin{ffact}\label{ffact2}
$\NN=\N^G\stackrel{\mathrm{def}}{=}\{x\in\N:\;gxg^{-1}=x\;\;\mbox{for all}\; g\in G\}$.
\end{ffact}
\begin{proof}
Clearly $\N^G\subseteq\NN$. Since $\N$ is a normal subgroup of $\A$
$$
[x,g]\in\N\cap G=1
$$
for all $x\in\NN$ and $g\in G$. Therefore we also have $\NN\subseteq\N^G$.
\end{proof}

Note that the action by left translations of $\NN$ on $\hat M$ descends to an action by ``translations" on $M$. Let $t$ be such a (left) translation on $M$ that is homotopic to $id_M$ and let $\hat t\colon\hat M\to\hat M$ be a lift of $t$ which is homotopic to $id_{\hat M}$. Then $\hat t$ is also 
a left translation. Clearly $\hat t\in\NN=\N^G$.
\begin{ffact}\label{ffact3}
There exists a lift $\tilde t\colon N\to N$ of $\hat t$ such that $\tilde ta\tilde t^{-1}=a$ for all $a\in\pi_1(M)$.
\end{ffact}
\begin{proof}
Consider the group $\widetilde \N^G$ of all lifts of elements of $\N^G$ to the universal cover $N$. Since $\N^G\cap G=1$  the group $\N^G$ acts effectively on $M$ and the sequence
$$
1\to\pi_1(M)\to\widetilde\N^G\to\N^G\to 1
$$
is exact. There is a homomorphism
$$
H\colon\N^G\to Out(\pi_1(M))
$$
induced by conjugation by elements of $\widetilde \N^G$. (See Section~IV.6 of~\cite{Br}.) 

Take any lift $\bar t\in\widetilde N^G$ of $\hat t$. Recall that $t$ is homotopic to identity. This implies that $H(t)=[id_{\pi_1M}]$. Therefore
the corresponding automorphism of $\pi_1(M)$
$$
a\mapsto \bar ta\bar t^{-1}
$$
is an inner automorphism
$$
a\mapsto b^{-1}ab
$$
where $b\in\pi_1(M)$. Then $\tilde t\stackrel{\mathrm{def}}{=}b\bar t$ is the posited lift.
\end{proof}

\begin{ffact}\label{ffact4}
Assume that $f\colon\hat M\to\hat M$ is an affine diffeomorphism of the nilmanifold $\hat M$ homotopic to $id_{\hat M}$. Then $f$ is a translation.
\end{ffact}
\begin{proof}
Let $\tilde f\colon N\to N$ be a lift of $f$. Then $\tilde f $ has the form $\tilde f(x)=v A(x)$, where $v\in N$ and $A$ is an automorphism of $N$. The automorphism $A$ restricts to an automorphism of $\Gamma$. And, since $f$ is homotopic to identity, $A(\gamma)=a\gamma a^{-1}$, where $\gamma\in\Gamma$ and $a\in \Gamma$ is fixed. By~\cite[Theorem~5]{Malcev} $A|_{\Gamma}$ extends uniquely to an automorphism of $N$. Hence
$$
\forall x\in N\;\; A(x)=a x a^{-1}
$$
and we see that $x\mapsto vax$ is a translation that covers $f$.
\end{proof}

We start the proof of Proposition~\ref{prop_translation}. Let $x_0$ be a fixed point of $L$. Denote by  $\hat L$ and $\hat x_0$ lifts of $L$ and $x_0$ to $\hat M$ respectively. Also denote by $\hat h$ a lift of $h$ which is homotopic to $id_{\hat M}$.

Notice that $h(x_0)$ is also fixed by $L$. Clearly $\hat x_0$ and $\hat h(\hat x_0)$ are periodic for $\hat L$ and, therefore, are fixed by some power $\hat L^q$. Thus $\hat h( \hat L^q (\hat x_0))=\hat L^q(\hat h(\hat x_0))$ for some $q>0$, but $\hat h\circ \hat L^q=g_0\circ \hat L^q\circ \hat h$ for some $g_0\in G$. 
 Hence, since $G$ acts freely, $g_0=id_{\hat M}$. Therefore $\hat L^q$ and $\hat h$ commute and Theorem~2 of Walter's paper~\cite{W} implies that $\hat h$ is an affine diffeomorphism of $\hat M$. By Fact~\ref{ffact4}, the affine diffeomorphism $\hat h$ must be a translation; that is, $\hat h\in\N$.
 
Clearly $\hat h$ normalizes $G$. Therefore, by Fact~\ref{ffact2}, $\hat h\in\N^G$ and, by Fact~\ref{ffact3}, $\hat h$ admits a lift $\tilde h$ in $N^{\pi_1M}$. Fact~\ref{ffact1} implies that there is a path that connects $\tilde h$ to $id_N$ in $N^{\pi_1M}$. This path projects to a path that connects $h$ and $id_M$. Thus $h$ is isotopic to identity.

\end{document}